\documentclass[12pt,reqno]{amsart}

\usepackage{amsthm,amssymb,amsmath}
\usepackage{amscd}
\usepackage[all]{xy}
\usepackage[dvipdfm]{graphicx} 
\usepackage[dvips]{color}
\usepackage{multirow, bigdelim}

\usepackage{comment}

\theoremstyle{plain}
\newtheorem{thm}{Theorem}[section]

\newtheorem{lemma}[thm]{Lemma}
\newtheorem{prop}[thm]{Proposition}

\newtheorem{rem}[thm]{Remark}
\newtheorem{claim}{Claim}

\makeatletter
\long\def\@makecaption#1#2{%
\vskip\abovecaptionskip%
\sbox\@tempboxa{#1 #2}
\ifdim \wd\@tempboxa >\hsize%
#1 #2\par
\else
\global \@minipagefalse
\hb@xt@\hsize{\hfil\box\@tempboxa\hfil}%
\fi
\vskip\belowcaptionskip}
\makeatother

\begin{document}

\title{Automorphism groups of smooth plane curves} 
\author{Takeshi Harui}
\date{June 7, 2014} 

\subjclass[2000]{Primary 14H45; Secondary 14H50, 14C21, 14H51}
\keywords{plane curves, automorphism groups}

\address{Takeshi Harui: Academic Support Center, Kogakuin University, 2665-1 Nakano, Hachioji, Tokyo 192-0015, Japan.}
\email{takeshi@cwo.zaq.ne.jp, kt13459@ns.kogakuin.ac.jp} 

\begin{abstract}
The author classifies finite groups acting on smooth plane curves of degree at least four.
Furthermore, he gives some upper bounds for the order of automorphism groups of smooth plane curves
and determines the exceptional cases in terms of defining equations.
This paper also contains a simple proof of the uniqueness of smooth plane curves 
with the full automorphism group of maximum order for each degree.
\end{abstract}

\maketitle


\section{Background  and Introduction} \label{Sec:Background}
The group of automorphisms of an algebraic curve defined over the complex number field is an old subject of research in algebraic geometry
and there are many works on the order of the group of automorphisms.
Among others, Hurwitz \cite{Hu} gave an universal upper bound (see Theorem \ref{thm:Hur} for the precise statement).
It is an application of Riemann-Hurwitz formula.
Following the same line, Oikawa \cite{O} proved another (and possibly better) upper bound for the order of automorphism groups with invariant subsets.
Later Arakawa \cite{A} proceeded further with a similar method (see Theorem \ref{thm:OA} for their works).
Their results are very useful for our study on smooth plane curves.

There are also many works for the structure of automorphism groups of algebraic curves.
In particular, the full automorphism groups of hyperelliptic curves are well known (\cite{BEM}, \cite{BGG}).
However, it seems that we still do not have enough knowledge about the determination of the full automorphism groups of non-hyperelliptic curves
except for some special cases, for example, the cases of low genus (\cite{Br}, \cite{He}, \cite{KKu}, \cite{KKi}) and Hurwitz curves. 

For plane curves, we have many examples of smooth plane curves whose group of automorphisms are completely known, such as Fermat curves (\cite{T}).
In the joint works with Komeda, Kato and Ohbuchi, the author gave a classification of smooth plane curves with automorphisms of certain type
(\cite{HKO}, \cite{HKKO}).
Bradley and D'Souza \cite{BD} gave upper bounds for the order of automorphism groups and collineation groups of singular plane curves in terms of their degree and the number of singularities.
There seems, however, no general result on the structure of automorphism groups of plane curves, even if they are smooth, as long as the author knows.

Automorphism groups of smooth plane curves of degree at most three is classically well known.
In this paper we study the cases of higher degree and consider the following problems:

\vspace{1em}

\noindent
\textbf{Problem.} (1) Classify automorphism groups of smooth plane curves.

\noindent
(2) Give a sharp upper bound for the order of automorphism groups of such curves.

\noindent
(3) Determine smooth plane curves with the group of automorphisms of large order.

\vspace{1em}

\noindent
We shall give a complete answer for each problem in Theorem \ref{thm:main1}, Theorem \ref{thm:main2} and Theorem \ref{thm:main3} respectively.
These are the main results of this article.

The first theorem, roughly speaking, states that smooth plane curves are divided into five kinds by their full automorphism group.
Curves of the first kind are smooth plane curves whose full automorphism group is cyclic.
The second kind consists of curves whose full automorphism group is the central extension of a finite subgroup of M\"{o}bius group $\textup{PGL}(2,\mathbb{C}) = \textup{Aut} (\mathbb{P}^1)$
by a cyclic group.
Curves of the third (resp.~the fourth) kind are descendants of Fermat (resp.~Klein) curves (see Section \ref{Sec:Main results} for the definition of this concept). 
For curves of the fifth kind, their full automorphism group is isomorphic to a primitive subgroup of $\textup{PGL}(3,\mathbb{C})$.  
It seems surprising that a smooth plane curve is a descendant of Fermat curve or Klein curve unless its full automorphism group is primitive or has a fixed point in the plane. 

There are several by-products of Theorem \ref{thm:main1} on automorphism groups of smooth plane curves.
We obtain, for example, a sharp upper bound of the order of such groups in Theorem \ref{thm:main2}.
For smooth plane curves, it is natural to expect that there exists a stronger upper bound of the order of their automorphism groups than Hurwitz's one.
Indeed, we show that the order of the full automorphism group of a smooth plane curve $d \ne 4,6$ is at most $6d^2$, which is attained by Fermat curve.
Moreover, a smooth plane curve with the full automorphism group of maximum order is unique for each degree up to projective equivalence.
We remark that Theorem \ref{thm:main2} has been shown in several special cases: $d =4$ (classical), $d=6$ (\cite{DIK}) and $d$ is a prime at most $20$ (\cite{KMP}).
Furthermore, Pambianco gave a complete proof of the same theorem for $d \ge 8$ in a different way (\cite[Theorem 1]{P}).

Our third main result, Theorem \ref{thm:main3}, is a classification of smooth plane curves with automorphism groups of large order in terms of defining equations.

\section{Main results} \label{Sec:Main results}

First of all, we note a simple fact on automorphism groups of smooth plane curves and introduce several concepts.
Let $G$ be a group of automorphisms of a smooth plane curve of degree at least four.
Then it is naturally considered as a subgroup of $\textup{PGL}(3,\mathbb{C}) = \textup{Aut} (\mathbb{P}^2)$.

Let $F_d$ be Fermat curve $X^d+Y^d+Z^d=0$ of degree $d$.
In this article we denote by $K_d$ the smooth plane curve defined by the equation $X Y^{d-1}+Y Z^{d-1}+Z X^{d-1}=0$, which is called \textit{Klein curve} of degree $d$.

For a non-zero monomial $cX^i Y^j Z^k$ we define its \textit{exponent} as $\textup{max} \{ i,j,k \}$.
For a homogeneous polynomial $F$, the \textit{core} of $F$ is defined as the sum of all terms of $F$ with the greatest exponent.
A term of $F$ is said to be \textit{low} if it does not belong to the core of $F$.

Let $C_0$ be a smooth plane curve of degree at least four.
Then a pair $(C,G)$ of a smooth plane curve $C$ and a subgroup $G \subset \textup{Aut} (C)$ is said to be a \textit{descendant} of $C_0$ 
if $C$ is defined by a homogeneous polynomial whose core is a defining polynomial of $C_0$ and $G$ acts on $C_0$ under a suitable coordinate system. 
We simply call $C$ a descendant of $C_0$ if $(C, \textup{Aut} (C))$ is a descendant of $C_0$.

In this article, we denote by $\textup{PBD}(2,1)$ the subgroup of $\textup{PGL}(3,\mathbb{C})$ 
that consists of all elements representable by a $3 \times 3$ complex matrix $A$ of the form
\[
\begin{pmatrix}
\multicolumn{2}{c}{\multirow{2}*{\Large $A'$}} & 0 \\
 & & 0 \\
0 & 0 & \alpha 
\end{pmatrix}
\left( \textup{$A'$ is a regular $2 \times 2$ matrix, $\alpha \in \mathbb{C}^*$} \right).
\]
There exists the natural group homomorphism $\rho : \textup{PBD}(2,1) \to \textup{PGL}(2,\mathbb{C})$ $([A] \mapsto [A'])$,
where $[M]$ denotes the equivalence class of a matrix $M$.
Using these concepts we state our first main result as follows:

\begin{thm} \label{thm:main1}
Let $C$ be a smooth plane curve of degree $d \ge 4$, $G$ a subgroup of $\textup{Aut}(C)$.
Then one of the following holds:
\begin{itemize}
\item[\textup{(a-i)}] $G$ fixes a point on $C$ and $G$ is a cyclic group whose order is at most $d(d-1)$.
Furthermore, if $d \ge 5$ and $|G| = d(d-1)$, then $C$ is projectively equivalent to the curve $YZ^{d-1} + X^d + Y^d =0$.
\item[\textup{(a-ii)}] $G$ fixes a point not lying on $C$ and there exists a commutative diagram
\begin{align*}
1 \to & \; \mathbb{C}^* \to \textup{PBD}(2,1) \stackrel{\rho} \to \textup{PGL}(2, \mathbb{C}) \to 1 \quad (\textup{exact}) \\
& \; \; \rotatebox{90}{$\hookrightarrow$} \hspace{3.6em} \rotatebox{90}{$\hookrightarrow$} \hspace{4.3em} \rotatebox{90}{$\hookrightarrow$} \\
1 \to & \;  N \; \; \longrightarrow \; \; G \; \; \; \longrightarrow \; \; \; G' \; \; \to 1 \quad (\textup{exact}), 
\end{align*}
where $N$ is a cyclic group whose order is a factor of $d$ and $G'$ is conjugate to a cyclic group $\mathbb{Z}_m$, a dihedral group $D_{2m}$, the tetrahedral group $A_4$, the octahedral group $S_4$ or the icosahedral group $A_5$, where $m$ is an integer at most $d-1$.
Moreover, if $G' \simeq D_{2m}$, then $m| \,d-2$ or $N$ is trivial.
In particular $|G| \le \textup{max} \{ 2d(d-2), 60d \}$.
\item[\textup{(b-i)}] $(C,G)$ is a descendant of Fermat curve $F_d:X^d+Y^d+Z^d=0$.
In this case $|G| \le 6d^2$.
\item[\textup{(b-ii)}] $(C,G)$ is a descendant of Klein curve $K_d: X Y^{d-1}+Y Z^{d-1}+Z X^{d-1}=0$.
In this case $|G| \le 3(d^2-3d+3)$ if $d \ge 5$.
On the other hand, $|G| \le 168$ if $d=4$.
\item[\textup{(c)}] $G$ is conjugate to a finite primitive subgroup of $\textup{PGL}(3,\mathbb{C})$, namely, the icosahedral group $A_5$,  
the Klein group $\textup{PSL}(2, \mathbb{F}_7)$, the alternating group $A_6$,
the Hessian group $H_{216}$ of order $216$ 
or its subgroup of order $36$ or $72$.
In particular $|G| \le 360$.
\end{itemize}
\end{thm}

We make some remarks on this theorem.

\begin{rem} \label{rem:main1}
\normalfont
\noindent
(1) In cases (a-i) and (a-ii), $G$ fixes a point, say $P$.
In fact $G$ also fixes a line not passing through $P$, which follows from Theorem \ref{thm:sub_PGL}.

\noindent
(2) A point $P$ in $\mathbb{P}^2$ is called a \textit{Galois point} for $C$ 
if the projection $\pi_P$ from $C$ to a line with center $P$ is a Galois covering.
A Galois point $P$ for $C$ is said to be \textit{inner} (resp.~\textit{outer}) if $P \in C$ (resp.~$P \not \in C$).
In the case\,(a-ii), if $|N| = d$ then the fixed point of $G$ is an outer Galois point for $C$.

\noindent
(3) The Klein group in the case\,(c) is the full automorphism group of Klein quartic
and the alternating group $A_6$ is that of Wiman sextic (see Theorem \ref{thm:main2}).
The Hessian group of order $216$ is generated by the four elements $h_i$ $(i=1,2,3,4)$ represented by the following matrices:
\[
\begin{pmatrix}
0 & 1 & 0 \\
0 & 0 & 1 \\
1 & 0 & 0
\end{pmatrix}, 
\begin{pmatrix}
1 & 0 & 0 \\
0 & \omega & 0 \\
0 & 0 & \omega^2
\end{pmatrix},
\begin{pmatrix}
1 & 1 & 1 \\
1 & \omega & \omega^2 \\
1 & \omega^2 & \omega
\end{pmatrix}
\; \textup{and} \; 
\begin{pmatrix}
1 & 0 & 0 \\
0 & \omega & 0 \\
0 & 0 & \omega
\end{pmatrix}, 
\]
where $\omega$ is a primitive third root of unity.
This group is the full automorphism group of a smooth plane sextic (see Remark \ref{rem:main2} (2)).
Its primitive subgroups of order $36$ and $72$ are respectively equal to $\langle h_1, h_2, h_3 \rangle$ and $\langle h_1, h_2, h_3, u \rangle$,
where $u = h_1^{-1} h_4^2 h_1$. 
\end{rem}

As a corollary of Theorem \ref{thm:main1}, we obtain a sharp upper bound for the order of automorphism groups of smooth plane curves
and classify the extremal cases.

\begin{thm} \label{thm:main2}
Let $C$ be a smooth plane curve of degree $d \ge 4$.
Then $|\textup{Aut}(C)| \le 6d^2$ except the following cases:
\begin{itemize}
\item[\textup{(i)}] $d=4$ and $C$ is projectively equivalent to Klein quartic $XY^3+YZ^3+ZX^3 =0$.
In this case $\textup{Aut}(C)$ is the Klein group $\textup{PSL}(2, \mathbb{F}_7)$, which is of order $168$. 
\item[\textup{(ii)}] $d=6$ and $C$ is projectively equivalent to the sextic 
\[ 10 X^3 Y^3 + 9 X^5Z + 9Y^5 Z - 45 X^2 Y^2 Z^2 - 135 X Y Z^4 + 27 Z^6 = 0. \]
In this case $\textup{Aut}(C)$ is equal to $A_6$, a simple group of order $360$.
\end{itemize}
Furthermore, for any $d \ne 6$, the equality $|\textup{Aut}(C)| = 6d^2$ holds if and only if $C$ is projectively equivalent to Fermat curve $F_d: X^d+Y^d+Z^d=0$, 
in which case $\textup{Aut} (C)$ is a semidirect product of $S_3$ acting on $\mathbb{Z}_d^2$.
In particular, for each $d \ge 4$, there exists a unique smooth plane curve with the full group of automorphisms of maximum order up to projective equivalence.
\end{thm}

\begin{rem} \label{rem:main2}
\normalfont
(1) It is classically known that Klein quartic has the Klein group $\textup{PSL}(2, \mathbb{F}_7)$ as its group of automorphisms
(see \cite{Bl}).
For the sextic in the above theorem, Wiman \cite{W} proved that its group of automorphisms is isomorphic to $A_6$.
In \cite{DIK} Doi, Idei and Kaneta called this curve \textit{Wiman sextic} and showed that it is the only smooth plane sextic whose full automorphism group has the maximum order $360$. 
We shall give a simpler proof on the uniqueness of Klein quartic (resp.~Wiman sextic) as a smooth plane curve of degree four (resp.~six) with the group of automorphisms of maximum order. 
\\
(2) When $d=6$, the smooth plane sextic defined by the equation
\[ X^6+Y^6+Z^6-10(X^3 Y^3 + Y^3 Z^3 + Z^3 X^3  ) =0 \]
also satisfies $|\textup{Aut}(C)|=216=6^3$.
In this case $\textup{Aut}(C)$ is equal to the Hessian group of order $216$, therefore this curve is not a descendant of Fermat curve.
\end{rem}

As another by-product of Theorem \ref{thm:main1}, we also give a stronger upper bound for the order of automorphism groups of smooth plane curves
and classify the exceptional cases when $d \ge 60$:

\begin{thm} \label{thm:main3}
Let $C$ be a smooth plane curve of degree $d \ge 60$.
Then $|\textup{Aut}(C)| \le d^2$ unless $C$ is projectively equivalent to one of the following curves:
\begin{itemize}
\item[\textup{(i)}] Fermat curve $F_d: X^d+Y^d+Z^d=0$ $(|\textup{Aut}(F_d)|=6d^2)$.
\item[\textup{(ii)}] Klein curve $K_d: X Y^{d-1}+Y Z^{d-1}+Z X^{d-1}=0$ $(|\textup{Aut}(K_d)|=3(d^2-3d+3))$.
\item[\textup{(iii)}] the smooth plane curve defined by the equation
\[ Z^d + XY(X^{d-2}+Y^{d-2}) = 0, \]
in which case $|\textup{Aut}(C)| = 2d(d-2)$.
\item[\textup{(iv)}] a descendant of Fermat curve defined by the equation
\[ X^{3m}+Y^{3m}+Z^{3m} -3 \lambda X^mY^mZ^m =0, \]
where $d=3m$ and $\lambda$ is a non-zero number with $\lambda^3 \ne 1$.
In this case $|\textup{Aut}(C)| = 2d^2$.
\item[\textup{(v)}] a descendant of Fermat curve defined by the equation 
\[ X^{2m}+Y^{2m}+Z^{2m}+\lambda(X^mY^m+Y^mZ^m+Z^mX^m) =0,\]
where $d=2m$ and $\lambda \ne 0, -1, \pm2$. 
In this case $|\textup{Aut}(C)| = 6m^2 = (3/2)d^2$.
\end{itemize}
\end{thm}


\section{Preliminary results} \label{Sec:Preliminary}

\noindent
\textbf{Notation and Conventions} 

We identify a regular matrix with the projective transformation represented by the matrix
if no confusion occurs.
When a planar projective transformation preserves a smooth plane curve, it is also identified with the automorphism obtained by its restriction to the curve.

We denote by $[H_1(X,Y,Z),H_2(X,Y,Z),H_3(X,Y,Z)]$ a planar projective transformation defined by $(X:Y:Z) \mapsto (H_1(X,Y,Z):H_2(X,Y,Z):H_3(X,Y,Z))$,
where $H_1$, $H_2$ and $H_3$ are homogeneous linear polynomials. 

A planar projective transformation of finite order is classically called a \textit{homology} 
if it is written in the form $[X,Y,\zeta Z]$ under a suitable coordinate system, where $\zeta$ is a root of unity.
A non-trivial homology fixes a unique line pointwise and a unique point not lying the line.
They are respectively called its \textit{axis} and \textit{center}. 

A \textit{triangle} means a set of three non-concurrent lines. 
Each line is called an \textit{edge} of the triangle.

The line defined by the equation $X=0$ (resp.~$Y=0$ and $Z=0$) will be denoted by $L_1$ (resp.~$L_2$ and $L_3$).
We also denote by $P_1$ (resp.~$P_2$ and $P_3$) the point $(1:0:0)$ (resp.~$(0:1:0)$ and $(0:0:1)$).

For a positive integer $m$, we denote by $\mathbb{Z}_m$ (resp.~$\mathbb{Z}_m^r$) a cyclic group of order $m$ (resp.~the direct product of $r$ copies of $\mathbb{Z}_m$). 

\vspace{1em}

In this section $C$ denotes a smooth irreducible projective curve of genus $g \ge 2$ defined over the field of complex numbers.
Then the full group of its automorphisms is a finite group and we have a famous upper bound of its order, which is known as \textit{Hurwitz bound}:

\begin{thm} \textup{(\cite{Hu})} \label{thm:Hur}
Let $G$ be a subgroup of $\textup{Aut}(C)$.
Then $|G| \le 84(g-1)$.
More precisely, 
\[ 
\frac{|G|}{g-1} = 84, \, 48, \, 40, \, 36, \, 30 \; \; \textup{or} \; \; \frac{132}{5} \quad \textup{or} \quad \frac{|G|}{g-1} \le 24. 
\]
\end{thm} 

Oikawa \cite{O} and Arakawa \cite{A} gave possibly stronger upper bounds under the assumption that $G$ fixes finite subsets of $C$
(not necessarily pointwise).
The following theorem is an application of Riemann-Hurwitz formula:

\begin{thm} \textup{(\cite[Theorem 1]{O}, \cite[Theorem 3]{A})} \label{thm:OA}
Let $G$ be a subgroup of $\textup{Aut}(C)$.
\begin{itemize}
\item[\textup{(1)}] $($Oikawa's inequality$)$ If $G$ fixes a finite subset $S$ of $C$, i.e., $GS=S$, with $|S|=k \ge 1$, then $|G| \le 12(g-1) + 6k$.
\item[\textup{(2)}] $($Arakawa's inequality$)$ If $G$ fixes three distinct finite subsets $S_i$ $(i=1,2,3)$ of $C$ with $|S_i|=k_i \ge 1$, then $|G| \le 2(g-1) + k_1+k_2+k_3$.
\end{itemize} 
\end{thm}

As an application of the former inequality, we can determine the full automorphism groups of Fermat curves and Klein curves.
For Fermat curves, Tzermias \cite{T} verified Weil's assertion on the structure of $\textup{Aut}(F_d)$ that the group is a semidirect product of $S_3$ acting on $\mathbb{Z}_d^2$
in characteristic zero.
We give another proof of this fact.

\begin{prop} \label{prop:F} 
Let $d$ be an integer with $d \ge 4$.
Then the full group of automorphisms of Fermat curve $F_d$ is generated by four transformations $[\zeta X,Y,Z]$, $[X, \zeta Y, Z]$, $[Y,Z,X]$ and $[X,Z,Y]$, 
where $\zeta$ is a primitive $d$-th root of unity. 
It is isomorphic to a semidirect product of $S_3$ acting on $\mathbb{Z}_d^2$, in other words, there exists a split short exact sequence of groups 
\[ 1 \to \mathbb{Z}_d^2 \to \textup{Aut}(F_d) \to S_3 \to 1. \] 
In particular $|\textup{Aut}(F_d)| = 6d^2$.
\end{prop}

\begin{proof}
Let $H$ be the subgroup of $\textup{Aut}(F_d)$ generated by four transformations $[\zeta X,Y,Z]$, $[X,\zeta Y,Z]$, $[Y,Z,X]$ and $[X,Z,Y]$. 
This is a semidirect product of $S_3$ acting on $\mathbb{Z}_d^2$. 
In particular we have the inequality $|\textup{Aut}(F_d)| \ge |H| = 6d^2$.
Thus it suffices to verify that $|\textup{Aut}(F_d)| \le 6d^2$.

Recall that Fermat curve $F_d$ has exactly $3d$ total inflection points.
They constitute a subset of $F_d$ fixed by its full group of automorphisms.
Hence it follows from Oikawa's inequality that
\[ |\text{Aut}(F_d)| \le 12(g-1) + 6 \cdot 3d = 6d^2, \]
where $g=(d-1)(d-2)/2$ is the genus of $F_d$.
\end{proof}

\begin{rem} \label{rem:ord_F}
\normalfont
It is easy to check that the order of any element of $\textup{Aut}(F_d)$ is at most $2d$.
\end{rem}

We also describe the structure of the full automorphism groups of Klein curves.
It is probably known, though the author could not find the explicit proof in the literature.

\begin{prop} \label{prop:K} 
If $d \ge 5$ then the full group of automorphisms of Klein curve 
\[ K_d: X Y^{d-1}+Y Z^{d-1}+Z X^{d-1}=0 \]
is generated by two transformations $[\xi^{-(d-2)} X,\xi Y,Z]$ and $[Y,Z,X]$,
where $\xi$ is a primitive $(d^2-3d+3)$-rd root of unity. 
It is isomorphic to a semidirect product of $\mathbb{Z}_3$ acting on $\mathbb{Z}_{d^2-3d+3}$, in other words, there exists a split short exact sequence of groups
\[ 1 \to \mathbb{Z}_{d^2-3d+3} \to \textup{Aut}(K_d) \to \mathbb{Z}_3 \to 1. \] 
In particular $|\textup{Aut}(K_d)| = 3(d^2-3d+3)$.
\end{prop}

\begin{proof}
Let $H$ be the subgroup of $\textup{Aut}(K_d)$ generated by two transformations $[Y,Z,X]$ and $[\xi^{-(d-2)} X,\xi Y,Z]$,
where $\xi$ is a primitive $(d^2-3d+3)$-rd root of unity.
This is a semidirect product of $\mathbb{Z}_3$ acting on $\mathbb{Z}_{d^2-3d+3}$. 
In particular $|\textup{Aut}(K_d)|$ is a multiple of $|H| = 3(d^2-3d+3)$.

It remains to show that $|\text{Aut}(K_d)| \le 3(d^2-3d+3)$.
Kato proved that Klein curve $K_d$ has exactly three $(d-3)$-inflection points $P_1=(1:0:0)$, $P_2=(0:1:0)$ and $P_3=(0:0:1)$ (see \cite[Lemma 2.3]{Ka}).
They constitute a subset of $K_d$ fixed by its full group of automorphisms.
It follows from Oikawa's inequality that
\[ |\text{Aut}(K_d)| \le 12(g-1) + 6 \cdot 3 = 6(d^2-3d+3), \]
where $g=(d-1)(d-2)/2$ is the genus of $K_d$.
Hence it is sufficient to verify that $\textup{Aut}(K_d)$ is of odd order. 

Suppose that $K_d$ has an involution $\iota$.
Then it fixes at least one $(d-3)$-inflection point.
Without loss of generality, we may assume that $\iota$ fixes $P_3$.
Then it also fixes the tangent line $L_2: Y=0$ to $K_d$ at $P_3$, the set of the remaining two $(d-3)$-inflection points $\{ P_1, P_2 \}$
and the set $\{ L_1, L_3 \}$.
Therefore $\iota = [\alpha X, \beta Y, Z]$ ($(\alpha, \beta) = (1, -1), (-1,1)$ or $(-1,-1)$) or $[\gamma Y, \gamma X, Z]$ $(\gamma = \pm1)$.
It is easy to check that such an involution does not fix $K_d$.
Hence $K_d$ has no involution, or equivalently, $\textup{Aut}(K_d)$ is of odd order.
\end{proof}

The following is a well-known classical result:

\begin{prop} \label{prop:cyclic}
If a subgroup $G$ of $\textup{Aut}(C)$ fixes a point on $C$, then $G$ is cyclic.
\end{prop}

For cyclic groups of automorphisms of smooth plane curves, we have the following lemma.

\begin{lemma} \label{lem:cyclic}
Let $C$ be a smooth plane curve of degree $d$, $G$ a cyclic subgroup of  $\textup{Aut}(C)$.
Then $|G| \le d(d-1)$. 
Furthermore, if $G$ is generated by a homology with center $P$, then $|G|$ is a factor of $d-1$ $($resp.~$d)$
if $P \in C$ $($resp.~$P \not \in C)$.
The equality $|G|=d-1$ $($resp.~$|G|=d)$ holds if and only if $P$ is an inner $($resp.~outer$)$ Galois point for $C$
and $G$ is the Galois group at the point.
\end{lemma}

\begin{proof}
Let $\sigma$ be a generator of $G$ and $g=(d-1)(d-2)/2$ the genus of $C$.
We may assume that $\sigma$ is represented by a diagonal matrix.
Then $G$ fixes each of three lines $L_1:X=0$, $L_2:Y=0$ and $L_3:Z=0$ and each of three points $P_1=(1:0:0)$, $P_2=(0:1:0)$ and $P_3=(0:0:1)$.
Set $S_i:=C \cap L_i$ for $i=1,2$ and $3$ and $V:= \{ P_1, P_2, P_3 \}$. 
Each $S_i$ is a non-empty subset of $C$ of order at most $d$ and is fixed by $G$.

There are three cases according to the intersection of $C$ and $V$.

\vspace{0.5em}

\noindent
Case\,(a) $|C \cap V| \ge 2$.
We may assume that $P_1, P_2 \in C$. 
Then at least one of the three sets $S_1 \setminus \{ P_2 \}$, $S_2 \setminus \{ P_1 \}$ and $S_3 \setminus \{ P_1, P_2 \}$ is a non-empty subset of $C$ of order at most $d-1$ and is fixed by $G$.
Hence it follows from Arakawa's inequality that 
\[
|G| \le 2(g-1) + 1 + 1 + d-1 = d(d-2) +1 < d(d-1).
\]

\vspace{0.5em}

\noindent
Case\,(b) $|C \cap V| = 1$. 
We may assume that $P_1 \in C$ and $P_2, P_3 \not \in C$.
Then either $S_2 \setminus \{ P_1 \}$ or $S_3 \setminus \{ P_1 \}$ is a non-empty subset of $C$ of order at most $d-1$ and is fixed by $G$.
Note that $G$ also fixes $\{ P_1 \}$ and $S_1$, each of which is distinct from the above sets.
By using Arakawa's inequality again 
\[
|G| \le 2(g-1) + (d-1) + 1 + d = d(d-1).
\]

\vspace{0.5em}

\noindent
Case\,(c) $C$ and $V$ are disjoint.
Then we may assume that $C$ is defined by a homogeneous polynomial whose core is $X^d+Y^d+Z^d$.
Take a diagonal matrix $\textup{diag} (\alpha, \beta, 1)$ representing $\sigma$.
Since $\sigma$ fixes the monomial $X^d+Y^d+Z^d$ up to a constant, we see that $\alpha^d = \beta^d = 1$, which implies that $\sigma^d =1$, i.e., $|G|$ is a factor of $d$. 

\vspace{0.5em}

Assume that $\sigma$ is a homology. 
Then we may assume that $\sigma = [X,Y,\zeta Z]$, where $\zeta$ is a root of unity.
Its center is $P_3$ and its axis is $L_3$.
Let $\pi: C \to C/G$ be the quotient map, $\pi_{P_3}: C \to \mathbb{P}^1$ the projection with center $P_3$ ($\pi_{P_3}((X:Y:Z)) = (X:Y)$).
Then $\psi: C/G \to \mathbb{P}^1$ ($[x] \mapsto \pi_{P_3}(x)$) is well-defined, where $[x]$ is the equivalence class of $x \in C$.
We thus have a commutative diagram 
\[
	\xymatrix{
    C \ar[rr]^{\hspace{-1em} \pi} \ar[rd]_{\pi_{P_3}} & & C/G \ar[ld]^{\psi} \\
	& \mathbb{P}^1. \ar@{}[u]|*+++{\circlearrowright} & \\
	}
\]
In particular $|G|=\textup{deg} \pi$ is a factor of $\textup{deg} \pi_{P_3}$,
which is equal to $d-1$ (resp.~$d$) if $P_3 \in C$ (resp. $P_3 \notin C$).

If $|G|=\textup{deg} \, \pi$, then $\pi$ coincides the quotient map, which implies that $P_3$ is a Galois point for $C$ and $G$ is the Galois group at $P_3$.
\end{proof}

Combining Proposition \ref{prop:cyclic} and Lemma \ref{lem:cyclic}, we obtain a characterization of 
the plane curve $YZ^{d-1} + X^d + Y^d =0$.

\begin{prop} \label{prop:F_{d,d-1}}
For $d \ge 5$, let $F_{d,d-1}$ be the smooth plane curve defined by the equation $YZ^{d-1} + X^d + Y^d =0$.
Then $\textup{Aut} (F_{d,d-1})$ is isomorphic to $\mathbb{Z}_{d(d-1)}$, a cyclic group of order $d(d-1)$. 
Moreover, $F_{d,d-1}$ is the only smooth plane curve of degree $d$ with an automorphism of order $d(d-1)$.
\end{prop}

\begin{proof}
First note that this curve has the unique inner Galois point $(0:0:1)$ (cf.~\cite[Theorem 4]{Y}), which is fixed by $\textup{Aut} (F_{d,d-1})$.
Hence $\textup{Aut} (F_{d,d-1})$ is cyclic by virtue of Proposition \ref{prop:cyclic}.
In particular $|\textup{Aut} (F_{d,d-1})| \le d(d-1)$ by Lemma \ref{lem:cyclic}.
On the other hand, $F_{d,d-1}$ has automorphisms $[\zeta_d X, Y, Z]$ and $[X,Y,\zeta_{d-1} Z]$, where $\zeta_k$ is a primitive $k$-th root of unity.
Thus $\textup{Aut} (F_{d,d-1})$ contains $\mathbb{Z}_d \times \mathbb{Z}_{d-1} \simeq \mathbb{Z}_{d(d-1)}$, which implies that $\textup{Aut} (F_{d,d-1}) \simeq \mathbb{Z}_{d(d-1)}$.

Let $C$ be a smooth plane curve of degree $d$ with an automorphism $\sigma$ of order $d(d-1)$.
We may assume that $\sigma$ is represented by a diagonal matrix and from the proof of Lemma \ref{lem:cyclic}
we may further assume that $P_3 \in C$ and $P_1, P_2 \not \in C$.
Then $C$ is defined by a homogeneous polynomial $F$ of the form $YZ^{d-1}+X^d+Y^d+\textup{(other terms)}$
after a suitable change of the coordinate system if necessary, because $C$ is smooth at $P_3$.

Assume that $\sigma = [\alpha X, \beta Y, Z]$.
Since it fixes $F$ up to a constant we see that $\beta = \alpha^d = \beta^d$,
which implies that $\beta = \zeta_{d-1}$ and $\alpha = \zeta_{d-1} \zeta_d$,
where $\zeta_k$ is a primitive $k$-th root of unity.
Then it is clear that $C$ has two automorphisms $[\zeta_d X, Y,Z]$ and $[X,Y,\zeta_{d-1} Z]$.

Let $X^i Y^j Z^k$ $(i+j+k=d)$ be any term of $F$ without its coefficient.
It is fixed by these automorphisms since they fixes $X^d+Y^d$.
Therefore $\zeta_d^i = \zeta_{d-1}^k =1$, which shows that $i \equiv 0$ (mod $d$) and $k \equiv 0$ (mod $d-1$).
Thus $(i,k) = (0,d-1)$, $(d,0)$ or $(0,0)$, or equivalently, $F=YZ^{d-1}+X^d+Y^d$.
Hence the conclusion follows.
\end{proof}

\begin{rem}
\normalfont
Kontogeorgis \cite{Ko} determined the group of automorphisms of the function field $F_{n,m}$ of the curve
$x^n+y^m+1=0$ in arbitrary characteristic $\ne 2, 3$ by using a different method.
\end{rem}

In the end of this section, we refer to a theorem on finite groups of planar projective transformations,
which is a basic tool to prove Theorem \ref{thm:main1}.

\begin{thm} \textup{(\cite[Section 1-10]{M}, \cite[Theorem 4.8]{DI})} \label{thm:sub_PGL}
Let $G$ be a finite subgroup of $\textup{PGL}(3,\mathbb{C})$.
Then one of the following holds:
\begin{itemize}
\item[\textup{(a)}] $G$ fixes a line and a point not lying on the line;  
\item[\textup{(b)}] $G$ fixes a triangle; or
\item[\textup{(c)}] $G$ is primitive and conjugate to 
the icosahedral group $A_5$,  
the Klein group $\textup{PSL}(2, \mathbb{F}_7)$ $($of order $168)$, the alternating group $A_6$,
the Hessian group $H_{216}$ of order $216$ 
or its subgroup of order $36$ or $72$.
\end{itemize}
\end{thm}


\begin{rem}
\normalfont
To be precise, Mitchell \cite{M} proved that $G$ fixes a point, a line or a triangle unless $G$ is primitive and isomorphic to a group as in the case\,(c).
In fact, the first two cases are equivalent.
Indeed, if $G$ fixes a point (resp.~a line) then $G$ also fixes a line not passing through the point (resp.~a point not lying the line). 
It is a direct consequence of Maschke's theorem in group representation theory.
Combining this fact with Mitchell's result we obtain the above theorem.
\end{rem}


\section{Automorphism groups of smooth plane curves: Case\,(A)} \label{Sec:Case(A)}

In what follows $C$ always denotes a smooth plane curve of degree $d \ge 4$ 
and let $G$ be a finite subgroup of $\textup{Aut}(C)$,
which is also considered as a subgroup of $\textup{PGL}(3,\mathbb{C})$.
We identify an element $\sigma$ of $G$ with the corresponding planar projective transformation, which is also denoted by $\sigma$.

The following two sections are wholly devoted to prove Theorem \ref{thm:main1}.
From Theorem \ref{thm:sub_PGL} there are three cases:
\begin{itemize}
\item[(A)] $G$ fixes a line and a point not lying on the line.
\item[(B)] $G$ fixes a triangle and there exists neither a line nor a point fixed by $G$.
\item[(C)] $G$ is primitive and conjugate to a group described in Theorem \ref{thm:sub_PGL}\null. 
\end{itemize}
Note that the last case leads us to the statement\,(c) in Theorem \ref{thm:main1}.
We argue the other cases one by one and consider the first case in this section.

\vspace{1em}

\noindent
\textbf{Case\,(A):} $G$ fixes a line $L$ and a point $P$ not lying on $L$.

If $P \in C$ then $\textup{Aut}(C)$ is cyclic by virtue of Proposition \ref{prop:cyclic}.
Hence (a-i) in Theorem \ref{thm:main1} follows from and Lemma \ref{lem:cyclic} and Proposition \ref{prop:F_{d,d-1}}.

In what follows we assume that $P \not \in C$ and prove that the statement\,(a-ii) or (b-i) in Theorem \ref{thm:main1} holds.  
We may further assume that $L$ is defined by $Z=0$ and $P=(0:0:1)$.
Then $G$ is a subgroup of $\textup{PBD}(2,1)$.
Let $\rho$ be the restriction of the natural map from $\textup{PBD}(2,1)$ to $\textup{PGL}(2,\mathbb{C})$.
Then there exists a short exact sequence of groups
\[ 1 \to N \to G \stackrel{\rho} \to G' \to 1, \] 
where $N=\textup{Ker} \rho$ and $G' = \textup{Im} \rho$.

\begin{claim} \label{claim:A}
The subgroup $N$ is a cyclic group whose order is a factor of $d$.
\end{claim}

\begin{proof}
For each element $\eta$ of $N$, there exists a unique diagonal matrix of the form $\textup{diag} (1,1,\zeta)$ that represents $\eta$.
Then the homomorphism $\varphi: N \to \mathbb{C}^*$ $(\eta \mapsto \zeta)$ is injective.
Hence $N$ is isomorphic to a finite subgroup of $\mathbb{C}^*$, which implies that $N$ is a cyclic group generated by a homology.
Our assertion on the order of $N$ follows from Lemma \ref{lem:cyclic}.
\end{proof}

Let $\eta=[X,Y,\zeta Z]$ be a generator of $N$, where $\zeta$ is a root of unity.
On the other hand, it is well known that $G'$, a finite subgroup of $\textup{PGL}(2,\mathbb{C})$, is isomorphic to $\mathbb{Z}_m$, $D_{2m}$, $A_4$, $S_4$ or $A_5$.

In what follows we assume that $G' \simeq \mathbb{Z}_m$ or $D_{2m}$
and give upper bounds for $m$.
There exists an element $\sigma$ such that $\rho(\sigma) = \sigma'$ is of order $m$.
Let $H=\langle \sigma \rangle$ be the cyclic subgroup of $G$ generated by $\sigma$. 
We see that $\sigma = [\alpha X ,\beta Y,Z]$, where $\alpha$ and $\beta$ are roots of unity such that $\alpha/\beta$ is a primitive $m$-th root of unity.
Then the fixed points of $\sigma$ on $L$ are $P_1=(1:0:0)$ and $P_2=(0:1:0)$.
If $G' \simeq \mathbb{Z}_m$, then $G$ is generated by $\eta$ and $\sigma$,
which implies that $G$ fixes two points $P_1$ and $P_2$.

Additionally, when $G' \simeq D_{2m}$, there exists an element $\tau$ such that $\sigma'$ and $\tau' = \rho(\tau)$ generate $G'$ with $\tau'^2=1$ and $\tau' \sigma' \tau' = \sigma'^{-1}$.
Then $G$ is generated by $\eta$, $\sigma$ and $\tau$.
In this case we may also assume that $\tau = [\gamma Y, \gamma X, Z]$ for some $\gamma$, a root of unity.

Let $F$ be a defining homogeneous polynomial of $C$ and $e_k$ the intersection multiplicity $i_{P_k} (C,L)$ of $C$ and $L$ at $P_k$ $(k=1,2)$.
Note that $e_1=e_2$ if $G'$ is a dihedral group. 

For the triviality of $N$, we have the following:

\begin{claim} \label{claim:B}
If $e_1 \ge 2$ or $e_2 \ge 2$, then $N$ is trivial. 
\end{claim}

\begin{proof}
Without loss of generality, we may assume that $e_1 \ge 2$. 
Then, since $C$ is smooth at $P_1=(1:0:0)$, its defining polynomial $F$ contains a term of the form $c X^{d-1}Z$ $(c \ne 0)$.
Then $F$ is written as
\[ F = X^{e_2} Y^{e_1} F_1(X,Y) + c X^{d-1}Z + \textup{(other terms)}, \]
where $F_1$ is a homogeneous polynomial of $X$, $Y$ such that neither $X$ nor $Y$ is its factor.
Therefore 
\[ F^{\eta} = X^{e_2} Y^{e_1} F_1(X,Y) + \zeta^{-1} c X^{d-1}Z + \textup{(other terms)}, \]
which implies that $\zeta =1$, since $F^{\eta}$ is equal to $F$ up to a constant.
That is to say, $N$ is trivial.
\end{proof}

Next we distinguish two subcases:

\begin{itemize}
\item[\textup{(A-1)}] $C \cap L$ contains a point distinct from $P_1$ and $P_2$.
\item[\textup{(A-2)}] $C \cap L \subset \{ P_1, P_2 \}$.
\end{itemize}

We give a simple remark on these assumptions.

\begin{rem} \label{rem:assumption}
\normalfont
If $G' \simeq \mathbb{Z}_m$, we may assume that the former one is the case. 
Indeed, suppose that $C \cap L \subset \{ P_1, P_2 \}$.
Then $G$ fixes each of $P_1$ and $P_2$ and at least one of them are lying on $C$,
that is to say, $G$ fixes a point on $C$.
Hence (a-i) in Theorem \ref{thm:main1} follows from the argument in the beginning of this case.
\end{rem}

\noindent
\textbf{Subcase\,(A-1):} $C \cap L$ contains a point $Q$ distinct from $P_1$ and $P_2$.

We show the following claim:

\begin{claim} \label{claim:C}
The order $m$ of $\sigma'$ divides $d-e_1-e_2$.
Furthermore, if $m=d$ then $(C,G)$ is a descendant of Fermat curve $F_d$.
\end{claim}

\begin{proof}
Suppose that $\sigma^j$ fixes $Q$ for some $j$.
Then it fixes three points on $L$, namely, $Q$, $P_1$ and $P_2$.
Hence it fixes $L$ pointwise, that is to say, $\sigma^j \in N$.
In other words, $\sigma'^j=1$, which shows that $m|j$.
On the other hand, it is obvious that $\sigma^m$ fixes $Q$.
It follows that the order of the orbit of $Q$ by $H$ is equal to $|H/\langle \sigma^m \rangle| = m$. 
Therefore we conclude that $m | \, d-e_1-e_2$ using B\'{e}zout's theorem.

Assume that $m=d$.
Then $e_1=e_2=0$, which implies that neither $P_1$ nor $P_2$ lies on $C$.
It follows that $C$ is defined by a polynomial whose core is $X^d+Y^d+Z^d$ under a suitable coordinate system.
Recall that $G$ is generated by $\eta$ and $\sigma$ (resp.~$\eta$, $\sigma$ and $\tau$) when $G' \simeq \mathbb{Z}_m$ (resp.~$G' \simeq D_{2m}$).
Hence every element of $G$ fixes the polynomial $X^d+Y^d+Z^d$ up to a constant, in other words, $G$ is a subgroup of $\textup{Aut}(F_d)$.
Thus $(C,G)$ is a descendant of Fermat curve $F_d$.
\end{proof}

We obtain the assertion of Theorem \ref{thm:main1} by using Claim \ref{claim:A}, Claim \ref{claim:B} and Claim \ref{claim:C} as follows.


First $N$ is a cyclic group whose order is a factor of $d$ by Claim \ref{claim:A}.
Furthermore, when $G' \simeq \mathbb{Z}_m$, the inequality $m \le d-1$ holds 
or $(C,G)$ is a descendant of Fermat curve $F_d$ by Claim \ref{claim:C}.
Hence (a-ii) or (b-i) in Theorem \ref{thm:main1} holds.
On the other hand, when $G' \simeq D_{2m}$, note that $e_1=e_2$.
Therefore combining Claim \ref{claim:B} with Claim \ref{claim:C} we come to the following conclusion.
\begin{itemize}
\item[\textup{(i)}] $m | \, d-2$ if $e_1=e_2=1$.
\item[\textup{(ii)}] $m \le d-4$ and $N$ is trivial if $e_1=e_2 \ge 2$.
\item[\textup{(iii)}] $(C,G)$ is a descendant of Fermat curve $F_d$ if $e_1=e_2=0$.
\end{itemize}
That is to say, (a-ii) or (b-i) in Theorem \ref{thm:main1} follows in this case also.
Thus we complete the proof in this subcase.

\vspace{1em}


\noindent
\textbf{Subcase\,(A-2):} $C \cap L \subset \{ P_1, P_2 \}$, or equivalently, $e_1+e_2=d$.

As we noted in Remark \ref{rem:assumption}, we may assume that $G' \simeq D_{2m}$.
Furthermore, it follows from our assumption that $e_1=e_2=d/2 \ge 2$,
which implies that $N$ is trivial by virtue of Claim \ref{claim:B}.

It remains to prove that $m \le d-1$.
In fact we can show the following claim.

\begin{claim} \label{claim:D}
$m| \, d-1$.
\end{claim}

\begin{proof}
Since $C$ passes through $P_1=(1:0:0)$ and $P_2=(0:1:0)$ and $C$ is smooth, 
$F$, the defining polynomial of $C$, contains two terms of the form $c X^{d-1}Z$ and $c' Y^{d-1}Z$ $(c,c' \ne 0)$.
We then have the following equalities:
\begin{align*}
 F &= c X^{d-1}Z + c' Y^{d-1}Z + \textup{(other terms)}, \\
 F^{\sigma} &= \alpha^{-(d-1)} c X^{d-1}Z + \beta^{-(d-1)} c' Y^{d-1}Z + \textup{(other terms)}.
\end{align*}
Since $\sigma$ preserves $F$ up to a constant, we obtain the equality $\alpha^{-(d-1)} = \beta^{-(d-1)}$.
In other words, $\sigma'^{d-1}=1$, which implies that $m| \,d-1$.
\end{proof}

Thus the assertion (a-ii) in Theorem \ref{thm:main1} holds in this subcase,
which completes our proof in Case\,(A).

\section{Automorphism groups of smooth plane curves: Case\,(B)} \label{Sec:Case(B)}
In this section we show the statement\,(b-i) or (b-ii) in Theorem \ref{thm:main1} holds in Case\,(B).  

\vspace{1em}

\noindent
\textbf{Case\,(B):} $G$ fixes a triangle $\Delta$ and there exists neither a line nor a point fixed by $G$.

We may assume that $\Delta$ consists of three lines $L_1:X=0$, $L_2:Y=0$ and $L_3:Z=0$.
Let $V$ be the set of vertices of $\Delta$, i.e., $V= \{P_1, P_2, P_3 \}$.
Then $G$ acts on $V$ transitively because otherwise $G$ fixes a line or a point, which conflicts with our assumption.
It follows that either $C$ and $V$ are disjoint or $C$ contains $V$.

Let $F$ be a defining homogeneous polynomial of $C$.
We note a trivial but useful observation:\\

\noindent
\textbf{Observation.}
\textit{Each element of $G$ gives a permutation of the set $\{ X,Y,Z \}$ of the coordinate functions up to constants. }\\

If $C$ contains $V$, we denote by $T_i$ the tangent line to $C$ at $P_i$ $(i=1,2,3)$.
Note that these lines are distinct and not concurrent by our assumption.
Furthermore, $G$ fixes the set $\{ T_1, T_2, T_3 \}$ and acts on it transitively.
Thus Case\,(B) is divided into three subcases:
\begin{itemize}
\item[(B-1)] $C$ and $V$ are disjoint.
\item[(B-2)] $C$ contains $V$ and each of  $T_i$'s $(i=1,2,3)$ is an edge of $\Delta$.
\item[(B-3)] $C$ contains $V$ and none of $T_i$'s $(i=1,2,3)$ is an edge of $\Delta$.
\end{itemize}

\vspace{1em}

\noindent
\textbf{Subcase\,(B-1):} $C$ and $V$ are disjoint.

We show that $(C,G)$ is a descendant of Fermat curve $F_d: X^d+Y^d+Z^d=0$ in this subcase.
By our assumption the defining polynomial $F$ of $C$ is of the form
\[ F = a X^d+ b Y^d+ c Z^d + \textup{(low terms)} \quad (a,b,c \ne 0). \]
Furthermore, we may assume that $a=b=c=1$ after a suitable coordinate change if necessary.
Then the core of $F$ is $X^d+Y^d+Z^d$, which is fixed by each element of $G$ up to a constant from the above observation.
It follows that $G$ also acts on Fermat curve $F_d$, in other words, $G$ is a subgroup of $\textup{Aut}(F_d)$.
Thus we conclude that $(C,G)$ is a descendant of $F_d$.

\vspace{1em}

\noindent
\textbf{Subcase\,(B-2):} $C$ contains $V$ and each $T_i$ $(i=1,2,3)$ is an edge of $\Delta$.

We show that $(C,G)$ is a descendant of Klein curve $K_d:XY^{d-1}+YZ^{d-1}+ZX^{d-1}=0$ in this subcase.
Without loss of generality we may assume that $T_1=L_3$, $T_2=L_1$ and $T_3=L_2$.
Then the defining polynomial $F$ of $C$ is of the form
\[ F = aXY^{d-1}+bYZ^{d-1}+cZX^{d-1} + \textup{(low terms)} \quad (a,b,c \ne 0). \]
Again we may assume that $a=b=c=1$ after a suitable coordinate change if necessary.
Then the core of $F$ is $XY^{d-1}+YZ^{d-1}+ZX^{d-1}$, which is fixed by each element of $G$ up to a constant from the above observation.
Hence $G$ also acts on Klein curve $K_d$, that is to say, $G$ is a subgroup of $\textup{Aut}(K_d)$.
Thus $(C,G)$ is a descendant of $K_d$.

\vspace{1em}

\noindent
\textbf{Subcase\,(B-3):} $C$ contains $V$ and no $T_i$ $(i=1,2,3)$ is an edge of $\Delta$.

We show that this subcase does not actually occur.

Let $V'= \{P'_1, P'_2, P'_3 \}$ be the set of the intersection points of $T_1$, $T_2$ and $T_3$, 
where $P'_i$ is the intersection point of $T_j$ and $T_k$ with $\{ i,j,k \} = \{ 1,2,3 \}$.
They are pairwise distinct because otherwise $T_1$, $T_2$ and $T_3$ are concurrent and the intersection point of them is fixed by $G$, which conflicts with our assumption.
Thus $T_1$, $T_2$ and $T_3$ constitute a triangle $\Delta'$, which is fixed by $G$ and $V'$ is the set of its vertices.
Furthermore, $V$ and $V'$ are disjoint by our assumption.

Any element $\sigma \in G$ can be written in the form $\sigma = [\alpha X_i, \beta X_j, \gamma X_k]$
with some constants $\alpha$, $\beta$ and $\gamma$, where $\{ i,j,k \} = \{ 1,2,3 \}$, $X_1=X$, $X_2=Y$ and $X_3=Z$.
Hence we have the natural homomorphism $\rho: G \to S_3$ defined by 
\[
\rho ( \sigma ) =  
\Big( \begin{array}{ccc}
1 & 2 & 3 \\
i & j & k \\
\end{array}
\Big). 
\]
Then $\textup{Im} \rho$ is isomorphic to $\mathbb{Z}_3$ or $S_3$, since there exists neither a line nor a point fixed by $G$.
We show that $\textup{Ker} \rho$ is trivial by using the following observation:

\begin{lemma} \label{lem:triangle}
Let $\sigma$ be a non-trivial planar projective transformation of finite order.
\begin{itemize}
\item[(i)] If $\sigma$ is a homology, then its fixed points consist of its center and all points on its axis.
In particular, every triangle whose set of vertices is pointwise fixed by $\sigma$ contains its center as a vertex.
\item[(ii)] If $\sigma$ is not a homology, then it fixes exactly three points.
In particular, there exists a unique triangle whose set of vertices is pointwise fixed by $\sigma$.
\end{itemize}
\end{lemma}

Let $\sigma$ be any element of $\textup{Ker} \rho$.
Then it can be written in the form $[\alpha X, \beta Y, Z]$ $(\alpha, \beta \ne 0)$.
Hence it fixes $V$ pointwise, which implies that it fixes each $T_i$ $(i=1,2,3)$.
In particular it fixes $V'$ also pointwise.
Then it follows from the above lemma that $\sigma$ is trivial.
Thus $\textup{Ker} \rho$ is trivial, or equivalently, $G \simeq \textup{Im} \rho \simeq \mathbb{Z}_3$ or $S_3$.

If $G$ is isomorphic to $\mathbb{Z}_3$, then $G$ fixes a line, which contradicts our assumption.
Thus $G$ is isomorphic to $S_3$.
Hence $G$ is generated by $\eta = [Y, Z, X]$ and another element $\tau$ of order two with $\tau \eta \tau = \eta^{-1}$ 
after a suitable coordinate change if necessary. 
Then we may assume that $\tau = [\omega Y, \omega^{-1}X, Z]$ $(\omega^3 =1)$.
Both $\eta$ and $\tau$ fixes the same point $(1:\omega^2:\omega)$.
Therefore $G$ also fixes this point, which conflicts with our assumption again.
It follows that this subcase is excluded.

Thus we complete the proof of Theorem \ref{thm:main1} thoroughly.


\section{Smooth plane curves with automorphism groups of large order} \label{Sec:LARGE_GROUP}
In this section we shall prove Theorem \ref{thm:main2} and Theorem \ref{thm:main3}.
First we consider primitive groups acting on smooth plane curves.

\begin{prop} \label{prop:prim}
Let $C$ be a smooth plane curve of degree $d \ge 4$, $G$ a finite subgroup of $\textup{Aut}(C)$.
If $G$ is primitive, then $|G| \le 6d^2$ except the following cases:
\begin{itemize}
\item[\textup{(i)}] $d=4$ and $C$ is projectively equivalent to Klein quartic $XY^3+YZ^3+ZX^3 =0$ and $G \simeq \textup{Aut}(K_4) \simeq \textup{PSL}(2, \mathbb{F}_7)$. 
\item[\textup{(ii)}] $d=6$ and $C$ is projectively equivalent to Wiman sextic $W_6$, which is defined by
\[ 10 X^3 Y^3 + 9 Z X^5 + 9Y^5Z - 45 X^2 Y^2 Z^2 - 135 X Y Z^4 + 27 Z^6 = 0 \]
and $G \simeq \textup{Aut}(W_6) \simeq A_6$.
\end{itemize}
\end{prop}

\begin{proof}
First note that $\textup{Aut}(C)$ is also primitive, which implies that $|G| \le |\textup{Aut}(C)| \le 360$ by Theorem \ref{thm:sub_PGL}. 
Hence $|G| < 6d^2$ if $d \ge 8$.

Assume that $d \le 7$.
If $d=5$ or $7$, then we have the inequality $|G| < 6d^2$ except for $(d, |G|) = (5,168)$, $(5,216)$, $(5,360)$ or $(7,360)$ again by Theorem \ref{thm:sub_PGL}\null.
It is easy to check by Theorem \ref{thm:Hur} that these four exceptional cases do not occur.

Assume that $d=6$.
If $|G| < 360$, then $|G| \le 216 = 6d^2$ by Theorem \ref{thm:sub_PGL}\null.
Suppose that $|G| = 360$ and $C$ is not projectively equivalent to Wiman sextic $W_6$. 
Since $G$ is conjugate to $A_6$, we may assume that $G$ acts on both $C$ and $W_6$.
It follows from B\'{e}zout's theorem that $C \cap W_6$ is a non-empty subset of $C$ of order at most $6^2=36$,
which is fixed by $G$.
Applying Oikawa's inequality we come to the conclusion that $360 = |G| \le 12 \cdot 9 +6 \cdot 36 = 324$, a contradiction. 

For $d=4$, we can deduce the uniqueness of the quartic with the full automorphism group of maximum order in the same way as above. 
\end{proof}

We show Theorem \ref{thm:main2} by using Theorem \ref{thm:sub_PGL} and Oikawa's inequality. 

\begin{proof}[Proof of Theorem \ref{thm:main2}]
We may assume that $\textup{Aut}(C)$ is not primitive by virtue of Proposition \ref{prop:prim}.
Then it follows from Theorem \ref{thm:sub_PGL} that $\textup{Aut}(C)$ fixes a line or a triangle.
First suppose that $\textup{Aut}(C)$ fixes a line $L$.
Then $S:=C \cap L$ is a non-empty set of order at most $d$, which is also fixed by $\textup{Aut}(C)$.
Applying Theorem \ref{thm:OA} (1) we obtain the inequality
\[ |\textup{Aut}(C)| \le 12(g-1)+6|S| \le 6d(d-3)+6d = 6d(d-2) < 6d^2. \]
Next suppose that $\textup{Aut}(C)$ fixes a triangle $\Delta$.
Then $C \cap \Delta$ is a non-empty set of order at most $3d$, which is also fixed by $\textup{Aut}(C)$.
Thus we have the inequality $|\textup{Aut}(C)| \le 6d^2$ by the same argument as above.

Finally assume that $|\textup{Aut}(C)| = 6d^2$ and $d \ne 6$. 
From Proposition \ref{prop:prim} and the above argument $\textup{Aut}(C)$ fixes a triangle and does not fix a line.
Then $C$ is a descendant of Fermat curve $F_d$ by virtue of Theorem \ref{thm:main1}.
Comparing the order of two groups we know that $G = \textup{Aut}(F_d)$.
Let $X^iY^jZ^k$ $(i+j+k=d)$ be a term of $F$ without its coefficient.
Note that $[\zeta X,Y,Z]$ and $[X,\zeta Y,Z]$ ($\zeta$ is a primitive $d$-th root of unity), which are elements of $G$, preserve $F$.
Hence they also preserve the monomial $X^iY^jZ^k$.
Then $\zeta^i=\zeta^j=1$, which implies that $(i,j,k)=(d,0,0)$, $(0,d,0)$ or $(0,0,d)$.
It follows that $F = X^d+Y^d+Z^d$.
\end{proof}

In the rest of this section we show Theorem \ref{thm:main3}. 
Before starting our proof, we determine the full automorphism groups of curves in three exceptional cases (iii), (iv) and (v) in the theorem.

\begin{prop} \label{prop:D_{2(d-2)}}
Assume that $d \ge 4$ and $C$ is the smooth plane curve defined by the equation
$Z^d + XY(X^{d-2}+Y^{d-2}) = 0$. 
\begin{itemize}
\item[\textup{(i)}] If $d \ne 4,6$, then $\textup{Aut}(C)$ is a central extension of $D_{2(d-2)}$ by $\mathbb{Z}_d$.
In particular $|\textup{Aut}(C)| = 2d(d-2)$.
\item[\textup{(ii)}] If $d=6$, $\textup{Aut}(C)$ is a central extension of $S_4$ by $\mathbb{Z}_6$.
In particular $|\textup{Aut}(C)| = 144$.
\item[\textup{(iii)}] If $d=4$, then $C$ is isomorphic to Fermat quartic $F_4$.
In particular $\textup{Aut}(C) \simeq \mathbb{Z}_4^2 \rtimes S_3$ $(|\textup{Aut}(C)|=96)$.
\end{itemize}
\end{prop}

\begin{proof}
First assume that $d \ge 5$ and $d \ne 6$.
Note that $G=\textup{Aut}(C)$ contains three elements $\sigma=[\xi X,\xi^{-(d-1)} Y,Z]$, $\tau=[Y,X,Z]$ and $\eta=[X,Y,\zeta Z]$,
where $\xi$ (resp.~$\zeta$) is a primitive $d(d-2)$-nd (resp.~$d$-th) root of unity.
Then $H=\langle \sigma, \tau, \eta \rangle$ is a subgroup of $\textup{PBD}(2,1)$.
This is a central extension of $H' = \langle \sigma', \tau' \rangle \simeq D_{2(d-2)}$ by $\langle \eta \rangle \simeq \mathbb{Z}_d$,
where $\sigma'$ (resp.~$\tau'$) is the image of $\sigma$ (resp.~$\tau$) by the natural homomorphism $\rho: \textup{PBD}(2,1) \to \textup{PGL}(2,\mathbb{C})$.

Next we note that $C$ has an outer Galois point $P=(0:0:1)$.
Since $\sigma$ is of order $d(d-2) > 2d$ for $d>4$, it follows from Remark \ref{rem:ord_F} that $C$ is not isomorphic to Fermat curve $F_d$.
Hence $P$ is the unique outer Galois point for $C$ (see \cite[Theorem 4', Proposition 5']{Y}).
In particular $G$ fixes $P$.
Then it follows from Remark \ref{rem:main1} that $G$ also fixes a line not passing through $P$,
which is $L_3: Z=0$ since it is the only line fixed by $H$. 
Thus $G \subset \textup{PBD}(2,1)$,
from which we have the short exact sequence 
\[ 1 \to N=\textup{Ker} \rho \to G \stackrel{\rho} \to G'=\textup{Im} \rho \to 1. \]
By virtue of Theorem \ref{thm:main1}, the kernel $N$ coincides with $\langle \eta \rangle \simeq \mathbb{Z}_d$.
On the other hand, $G'$ is a finite subgroup of $\textup{PGL}(2,\mathbb{C})$ containing $H' \simeq D_{2(d-2)}$.
Hence $G'=H'$ or $G'$ isomorphic to $A_4$, $S_4$ or $A_5$ again by Theorem \ref{thm:main1}.
We show that $G'=H'$ by excluding the latter case.

Suppose that $G'$ isomorphic to $A_4$, $S_4$ or $A_5$.
Since $G$ fixes the line $L:Z=0$, the set $S = C \cap L$ is a non-empty subset of $C$ with $|S| \le d$.
It follows from Oikawa's inequality that 
\[ |G| \le 12(g-1) + 6 \cdot d = 6d(d-2), \quad (\ast) \] 
where $g=(d-1)(d-2)/2$ is the genus of $C$.

The order of an element of $G'$ is at most four (resp.~five) if $G' \simeq A_4$ or $S_4$ (resp.~$G' \simeq A_5$).
On the other hand,  $\textup{ord} \sigma' = d-2$.
It follows that $d=5$ (resp.~$d \le 7$) if $G' \simeq A_4$ or $S_4$ (resp.~$G' \simeq A_5$).

If $G' \simeq A_5$, then $60d = |G| \le 6d(d-2)$ from $(\ast)$, which implies that $d \ge 12$, a contradiction.

If $d=5$ and $G' \simeq S_4$, then $24 \cdot 5 = |G| \le 6 \cdot 5 \cdot 3$ from $(\ast)$ again, which is absurd.

If $d=5$ and $G' \simeq A_4$, then $H' \simeq D_6$ is isomorphic to a subgroup of $A_4$ of index two, which is impossible since $A_4$ has no such subgroup.
Thus we exclude this case.

Next assume that $d=6$.
We prove that $G=\textup{Aut}(C)$ is a central extension of $S_4$ by $\mathbb{Z}_6$.
It suffices to show that $G' \simeq S_4$.
Since $G'$ contains $H'$, a subgroup of order eight, $G'$ cannot be $A_4$.
Then we only have to find an element of $G'$ of order three for verifying that $G' \simeq S_4$.
Converting slightly the defining polynomial of $C$, we may assume that $C$ is defined by $Z^6-XY(X^4-Y^4)=0$.
Then it is easy to verify that $G'$ has an element of order three.
Indeed, a $3 \times 3$ matrix
\[
\begin{pmatrix}
\multicolumn{2}{c}{\multirow{2}*{\Large $A$}} & 0 \\
 & & 0 \\
0 & 0 & 1 
\end{pmatrix}
\left(
A = c 
\begin{pmatrix}
1 &  \sqrt{-1} \\
1 & -\sqrt{-1}
\end{pmatrix}
\right)
\]
gives an automorphism $\epsilon$ of $C$ for a suitable constant $c$.
Then $\epsilon' = \rho(\epsilon)$ is of order three.

Finally assume that $d=4$.
Set $F=Z^4+XY(X^2+Y^2)$.
Substituting $X+\sqrt{-1}Y$ (resp.~$X-\sqrt{-1}Y$) for $X$ (resp.~$Y$), $F$ is converted to 
\begin{align*}
 \tilde F &= Z^4 + (X+\sqrt{-1}Y)(X-\sqrt{-1}Y) ((X+\sqrt{-1}Y)^2 + (X-\sqrt{-1}Y)^2) \\
    &= Z^4 + (X^2+Y^2) \cdot 2(X^2-Y^2) \\
    &= Z^4 + 2 (X^4 - Y^4).
\end{align*}
Then it is clear that the curve defined by $\tilde F$ is isomorphic to Fermat quartic $F_4$.
\end{proof}

\begin{prop} \label{prop:F'_d}
For a positive integer $d=3m$, let $F'_d$ be a smooth plane curve defined by 
\[ X^{3m}+Y^{3m}+Z^{3m} -3 \lambda X^mY^mZ^m =0, \]
where $\lambda$ is a non-zero number with $\lambda^3 \ne 1$.
It is a descendant of Fermat curve $F_d$ and $\textup{Aut}(F'_d)$ is generated by five transformations $[\zeta^3 X,Y,Z]$, $[X, \zeta^3 Y, Z]$, $[\zeta X, \zeta^{-1} Y, Z]$, $[Y,Z,X]$ and $[X,Z,Y]$,
where $\zeta$ is a primitive $d$-th root of unity. 
In this case $|\textup{Aut}(C)| = 2d^2$.
\end{prop}

\begin{proof}
Let $H$ be a subgroup of $G=\textup{Aut}(F'_d)$ generated by five transformations $[\zeta^3 X,Y,Z]$, $[X, \zeta^3 Y, Z]$, $[\zeta X, \zeta^{-1} Y, Z]$, $[Y,Z,X]$ and $[X,Z,Y]$. 
Note that $H$ also acts on Fermat curve $F_d$.
It is easy to check that $|H| = 3m^2 \cdot 6 = 2d^2$, which divides the order of $G$.
On the other hand, $|G|$ is a proper factor of $6d^2$ from Theorem \ref{thm:main2}.
Thus $|G| = 2d^2$, which implies that $G=H$.
In particular $C$ is a descendant of Fermat curve $F_d$.
\end{proof}

\begin{prop} \label{prop:F''_d}
For a positive even integer $d=2m \ge 8$, let $F''_d$ be a smooth plane curve defined by 
\[ X^{2m}+Y^{2m}+Z^{2m}+\lambda(X^mY^m+Y^mZ^m+Z^mX^m) =0, \]
where $\lambda \ne 0, -1, \pm2$.
It is a descendant of Fermat curve $F_d$ and $\textup{Aut}(F''_d)$ is generated by four transformations $[\zeta^2 X,Y,Z]$, $[X, \zeta^2 Y, Z]$, $[Y,Z,X]$ and $[X,Z,Y]$,
where $\zeta$ is a primitive $d$-th root of unity. 
It is isomorphic to a semidirect product of $S_3$ acting on $\mathbb{Z}_m^2$, in other words, there exists a split short exact sequence of groups 
\[ 1 \to \mathbb{Z}_m^2 \to \textup{Aut}(F''_d) \to S_3 \to 1. \] 
In particular $|\textup{Aut}(F'_d)| =  6m^2 = (3/2)d^2$.
\end{prop}

\begin{proof}
Let $H$ be a subgroup of $G=\textup{Aut}(F''_d)$ generated by four transformations 
$[\zeta^2 X,Y,Z]$, $[X, \zeta^2 Y, Z]$, $[Y,Z,X]$ and $[X,Z,Y]$. 
This is a semidirect product of $S_3$ acting on $\mathbb{Z}_m^2$. 
In particular $|G| $ is divided by $|H| = 6m^2$.
Then it is easy to verify that $G$ is not isomorphic to any group in Theorem \ref{thm:main1} (c), 
since $m \ge 4$ and $H$ is not isomorphic to the Hessian group $H_{216}$.
Furthermore, since $H$ fixes no point, neither does $G$.
Thus we conclude that $F''_d$ is a descendant of Fermat curve $F_d$ or Klein curve $K_d$ by using Theorem \ref{thm:main1}.
Since $G$ has an even order, the latter is not the case.
Hence $F''_d$ is a descendant of Fermat curve $F_d$.

Suppose that $G$ contains an element of $\textup{Aut}(F_d)$ outside $H$.
Then it can be converted by $H$ to the transformation $[\zeta X, Y, Z]$.
This transformation, however, does not act on $C$, which shows that $G = H$.
\end{proof}

Next we classify descendants of Fermat curve with automorphism groups of large order.
Let $d$ be an integer at least four, $\zeta$ a primitive $d$-th root of unity.
In what follows we denote the projective transformations $[\zeta X,Y,Z]$, $[X,\zeta Y,Z]$ and $[X,Y,\zeta Z]$
by $\eta_1$, $\eta_2$ and $\eta_3$, respectively.

\begin{lemma} \label{lem:Aut_Fermat}
Let $C$ be a descendant of Fermat curve $F_d$ $(d \ge 4)$ and $G=\textup{Aut}(C)$.
Then there exists a commutative diagram
\begin{align*}
1 \to \; & \mathbb{Z}_d \times \mathbb{Z}_d \to \textup{Aut}(F_d) \stackrel{\rho} \longrightarrow S_3 \; \to 1 \quad (\textup{exact}) \\
& \quad \; \; \rotatebox{90}{$\hookrightarrow$} \hspace{4.2em} \rotatebox{90}{$\hookrightarrow$} \hspace{4.3em} \rotatebox{90}{$\hookrightarrow$} \\
1 \to \; & \quad \; H \; \; \longrightarrow \; \; \; \; G \; \; \longrightarrow \; \; \; \; G' \; \to 1 \quad (\textup{exact}), 
\end{align*}
where $H = \textup{Ker} (\rho|_G)$ and $G' = \textup{Im} (\rho|_G)$.
\begin{itemize}
\item[(1)] If $G$ contains two of three $\eta_t$'s then it contains the other and $C$ is projectively equivalent to $F_d$.
\item[(2)] If $G'$ is of order at least three and $G$ contains an $\eta_t$ for some $t$ $(1 \le t \le 3)$, 
then $G$ contains all $\eta_t$'s and $C$ is projectively equivalent to $F_d$.
\end{itemize}
\end{lemma}

\begin{proof}
(1) It is clear that $G$ contains all $\eta_t$'s.
Take a defining polynomial of $C$ whose core is $X^d+Y^d+Z^d$.
Let $X^iY^jZ^k$ $(i+j+k=d)$ be any term of the polynomial without its coefficient.
Since it is invariant for the action of each $\eta_t$, we have the equality $\zeta^i=\zeta^j=\zeta^k=1$,
or equivalently $(i,j,k)=(d,0,0)$, $(0,d,0)$ or $(0,0,d)$, which implies the assertion.

\noindent
(2) Assume that $|G'| \ge 3$ and $G$ contains an $\eta_t$ for some $t$ $(1 \le t \le 3)$.
We may assume that $t=1$ without loss of generality.
Noting that $G' \simeq \mathbb{Z}_3$ or $S_3$ since $G'$ is a subgroup of $S_3$,
we may further assume that $G$ contains an element of order three represented by a projective transformation
$[\zeta^a Y, \zeta^b Z, X]$,
where $a$ and $b$ are integers with $0 \le a, b <d$.
Then $G$ contains $[Y, \zeta^b Z,X]$ since $\eta_1 \in G$, which implies that
$(\eta_1^{-1})^b [Y, \zeta^b Z, X]^2 (\eta_1^{-1})^b = [Z, X, Y] \in G$.
Thus $G$ contains $\eta_1$ and $[Z, X, Y]$, which implies that it also contains $\eta_2$.
In particular $C$ is projectively equivalent to $F_d$ from (1).
\end{proof}

By using the above lemma we obtain a characterization of two descendants $F'_d$ and $F''_d$ of Fermat curve.

\begin{lemma} \label{lem:descendant}
For $d \ge 8$, two curves $F'_d$ and $F''_d$ are the only descendants of Fermat curve $F_d$ whose group of automorphisms has order greater than $d^2$ 
up to projective equivalence, except $F_d$ itself.
\end{lemma}

\begin{proof}
Let $C$ be a descendant of $F_d$ such that $C$ is not isomorphic to $F_d$ and $G=\textup{Aut}(C)$ has order greater than $d^2$.
Then $G$ is a proper subgroup of $\textup{Aut}(F_d)$ from Theorem \ref{thm:main2},
which implies that $|G| = 3d^2, 2d^2, (3/2)d^2$ or $(6/5)d^2$.
Recall the commutative diagram in Lemma \ref{lem:Aut_Fermat}:
\begin{align*}
1 \to & \mathbb{Z}_d \times \mathbb{Z}_d \to \textup{Aut}(F_d) \stackrel{\rho} \longrightarrow S_3 \; \to 1 \quad (\textup{exact}) \\
& \quad \; \; \rotatebox{90}{$\hookrightarrow$} \hspace{4.2em} \rotatebox{90}{$\hookrightarrow$} \hspace{4.3em} \rotatebox{90}{$\hookrightarrow$} \\
1 \to & \quad \; H \; \; \longrightarrow \; \; \; G \; \; \longrightarrow \; \; \; \; G' \; \to 1 \quad (\textup{exact}), 
\end{align*}
where $H = \textup{Ker} (\rho|_G)$ and $G' = \textup{Im} (\rho|_G)$.

First we show that $G'$ coincides with $S_3$.
Indeed, if $G'$ is a proper subgroup of $S_3$, then $G' \simeq \mathbb{Z}_2$ or $\mathbb{Z}_3$.

If $G' \simeq \mathbb{Z}_2$ then $H = \mathbb{Z}_d \times \mathbb{Z}_d$ since $|G| >d^2$,
which implies that $C$ is projectively equivalent to $F_d$ from Lemma \ref{lem:Aut_Fermat} (1), a contradiction.

If $G' \simeq \mathbb{Z}_3$ then $H = \mathbb{Z}_d \times \mathbb{Z}_d$ or $H$ is a subgroup of $\mathbb{Z}_d \times \mathbb{Z}_d$ of index two.
In the former case $C$ is projectively equivalent to $F_d$ again from Lemma \ref{lem:Aut_Fermat} (1), a contradiction. 
In the latter case $\sigma^2 \in H$ for any $\sigma \in \mathbb{Z}_d \times \mathbb{Z}_d$, which implies that
$H$ contains a subgroup $\langle \eta_1^2, \eta_2^2 \rangle$, which is of order $(d/2)^2=d^2/4$.
Therefore there exists an extra element of $H$, which is written as $\eta_1^a \eta_2^b$, where $a$ or $b$ is odd.
Then $H$ also contains $\eta_1$, $\eta_2$ or $\eta_1 \eta_2 = \eta_3^{-1}$, or equivalently, $H$ contains at least one $\eta_t$,
which implies a contradiction from Lemma \ref{lem:Aut_Fermat} (2).
Thus $G'$ coincides $S_3$.

Next we consider the group $H$, which is a subgroup of $\mathbb{Z}_d \times \mathbb{Z}_d$ of index less than six.
Let $F$ be a defining homogeneous polynomial of $C$.
We may assume that $F$ is written as 
\[ 
F = X^d+Y^d+Z^d+ \textup{(low terms)}. 
\]
Let $X^i Y^j Z^k$ $(i+j+k=d$, $0 \le i, j, k < d)$ be any low term of $F$ without its coefficient.

Consider two projections $\varpi_1 : \mathbb{Z}_d^{(1)} \times \mathbb{Z}_d^{(2)} \to \mathbb{Z}_d^{(1)} $ $( \eta_1^a \eta_2^b \mapsto \eta_1^a)$ and $\varpi_2 : \mathbb{Z}_d^{(1)} \times \mathbb{Z}_d^{(2)} \to \mathbb{Z}_d^{(2)} $ $( \eta_1^a \eta_2^b \mapsto \eta_2^b)$
and their restriction $\varpi_1|_H : H \to \mathbb{Z}_d^{(1)} $ and $\varpi_1|_H : H \to \mathbb{Z}_d^{(2)} $.
There are two cases:

\begin{itemize}
\item[\textup{(i)}] $\varpi_1|_H$ or $\varpi_2|_H$ is surjective.
\item[\textup{(ii)}] Neither $\varpi_1|_H$ nor $\varpi_2|_H$ is surjective.
\end{itemize}

\noindent
Case\,(i) We may assume that $\varpi_1|_H$ is surjective.
In this case $H$ contains an element $[\zeta X, \zeta^e Y, Z]$ ($0 \le e <d$).
It fixes the polynomial $X^i Y^j Z^k$, which implies that $i+ej \equiv 0$ (mod $d$).
Since $G' \simeq S_3$, we also see that $j+ek \equiv k+ei \equiv j+ei \equiv 0$ (mod $d$).
It follows from these congruences that $i \equiv j \equiv k$ (mod $d$), which implies that $i=j=k=d/3$.
Therefore $F = X^{3m}+Y^{3m}+Z^{3m}+ cX^mY^mZ^m$, where $d =3m$ and $c \ne 0$.
Furthermore, we can write $c = -3 \lambda$, where $\lambda$ is a non-zero number with $\lambda^3 \ne 1$ because $C$ is non-singular.
Thus we see that $C = F'_d$ in this case.

\vspace{1em}

\noindent
Case\,(ii) Neither $\varpi_1|_H$ nor $\varpi_2|_H$ is surjective.
Note that $|\textup{Ker} (\varpi_i|_H)| \le d/2$ $(i=1,2)$ since $H$ does not contain $\eta_1$ or $\eta_2$.
Then we see that $|\textup{Ker} (\varpi_1|_H)| = |\textup{Im} (\varpi_1|_H)| = |\textup{Ker} (\varpi_2|_H)| = |\textup{Im} (\varpi_2|_H)| = d/2$
because $H$ is a subgroup of $\mathbb{Z}_d \times \mathbb{Z}_d$ of index less than six.
In particular $d$ is an even integer, say $d=2m$ and $|H| = m^2$.
We also see that $\textup{Ker} (\varpi_1|_H)$ is a subgroup of $\mathbb{Z}_d = \langle \eta_1 \rangle$ of index two, that is to say, $\textup{Ker} (\varpi_1|_H) = \langle \eta_1^2 \rangle$.
Thus $H$ contains $\eta_1^2$.
In the same way we can show that $H$ also contains $\eta_2^2$.
Then $H$ contains the subgroup $\langle \eta_1^2, \eta_2^2 \rangle$,
which implies that $H= \langle \eta_1^2, \eta_2^2 \rangle$ since both of them have the same order $m^2$.

Both $\eta_1^2$ and $\eta_2^2$ fixes the monomial $X^i Y^j Z^k$.
It follows that $2i \equiv 2j \equiv 0 $ (mod $d$), which implies that $(i,j,k) = (m,m,0)$, $(m,0,m)$ or $(0,m,m)$.
Then it is easy to show that $F$ can be written as $X^{2m}+Y^{2m}+Z^{2m}+\lambda(X^mY^m+Y^mZ^m+Z^mX^m)$ $(\lambda \ne 0)$.
Note that $\lambda \ne -1, \pm2$ since $C$ is non-singular.
We thus conclude that $C =F''_d$ in this case.
\end{proof}

We also need to show the uniqueness of smooth plane curve of degree $d$ whose full automorphism group is of order $3(d^2-3d+3)$.

\begin{prop} \label{prop:uniquness_of_K_d}
Let $C$ be a smooth plane curve of degree $d \ge 5$, $G$ a subgroup of $\textup{Aut}(C)$.
Assume that $|G| = 3(d^2-3d+3)$. 
Then $C$ is projectively equivalent to Klein curve $K_d$ and $G=\textup{Aut}(K_d)$. 
\end{prop}

\begin{proof}
Note that $|G| = 3(d^2-3d+3)$ is an odd integer greater than $d^2$.
Hence (b-ii) in Theorem \ref{thm:main1} only can occur.
Thus $(C,G)$ is a descendant of Klein curve $K_d$.
Then $C$ is defined by a homogeneous polynomial whose core is $XY^{d-1}+YZ^{d-1}+ZX^{d-1}$
under a suitable coordinate system.
Furthermore $G=\textup{Aut}(K_d)$, since $|G|=3(d^2-3d+3)=|\textup{Aut}(K_d)|$.
Then $G$ contains an element $\sigma = [\xi^{-(d-2)} X,\xi Y,Z]$ ($\xi$ is a primitive $(d^2-3d+3)$-rd root of unity) from Proposition \ref{prop:K}.

Suppose that $F$ contains a low term $cX^iY^jZ^k$ $(c \ne 0, i+j+k=d)$.
We then have following equalities:
\begin{align*}
 F &= XY^{d-1}+YZ^{d-1}+ZX^{d-1} + cX^iY^jZ^k + \textup{(other low terms)}, \\
 F^{\sigma} &= \xi^{-1} (XY^{d-1}+YZ^{d-1}+ZX^{d-1}) + \xi^{(d-2)i-j} cX^iY^jZ^k + \textup{(other low terms)}.
\end{align*}
Since they are equal up to a constant we see that $\xi^{(d-2)i-j} = \xi^{-1}$, which implies that $\xi^{(d-2)i-j+1} =1$.
The indices $i$, $j$ and $k$ are at most $d-2$ since the term $c X^iY^jZ^k$ is low.
Hence $-d+3 \le (d-2)i-j+1 < (d-2)(d-1)+1=d^2-3d+3$, which shows that $(d-2)i-j+1 = 0$, i.e., $j = (d-2)i +1$.
Then we see that $i=0$, $j=1$ and $k=d-1$, which conflicts with our assumption that the term $c X^iY^jZ^k$ is low.
Thus we conclude that $F = XY^{d-1}+YZ^{d-1}+ZX^{d-1}$.
\end{proof}

Now we are ready to give a proof of Theorem \ref{thm:main3}.

\begin{proof}[Proof of Theorem \ref{thm:main3}]
Let $C$ be a smooth plane curve of degree $d \ge 60$, $F$ a defining homogeneous polynomial of $C$.
Assume that a subgroup $G$ of $\textup{Aut}(C)$ is of order greater than $d^2$.

Since $d \ge 60$, we have the inequalities $|G| > 60d$ and $|G| > 360$.
Then there are only three possibilities from Theorem \ref{thm:main1}:

\begin{itemize}
\item[\textup{(i)}] $G$ fixes a point $P$ not lying on $C$ and $G$ is isomorphic to a central extension of $D_{2(d-2)}$ by $\mathbb{Z}_d$.
\item[\textup{(ii)}] $(C,G)$ is a descendant of Fermat curve $F_d: X^d+Y^d+Z^d=0$.
\item[\textup{(iii)}] $(C,G)$ is a descendant of Klein curve $K_d: X Y^{d-1}+Y Z^{d-1}+Z X^{d-1}=0$.
\end{itemize}

\vspace{1em}

\noindent
Case\,(i) In this case $G$ also fixes a line $L$ not containing $P$.
We may assume that $P = (0:0:1)$ and $L$ is defined by $Z=0$.
Then $G$ is generated by three elements $\eta = [X,Y,\zeta Z]$, $\sigma = [X, \omega Y, \omega' Z]$ and $\tau = [\gamma Y, \gamma X, Z]$,
where $\zeta$, $\omega$, $\omega'$ and $\gamma$ are certain roots of unity and the order of $\zeta$ (resp.~$\omega$) is $d$ (resp.~$d-2$).
Since $\eta$ preserves $F$ up to a constant, $F$ is written as $F = Z^d + \hat F(X,Y)$, 
where $\hat F(X,Y)$ is a homogeneous polynomial of $X$ and $Y$ without multiple factors.
Furthermore, $C$ intersects $L$ transversally at $P_1=(1:0:0)$ and $P_2=(0:1:0)$ respectively, by virtue of Claim \ref{claim:B} in Section \ref{Sec:Case(A)}.
Hence $\hat F(X,Y)$ has a factor of the form $X- cY$ $(c \ne 0)$.
Since $\sigma$ preserves $\hat F(X,Y)$ up to a constant, 
we conclude that $\hat F(X,Y) = \lambda XY\Pi_{k=0}^{d-3} (X-\omega^k cY) = \lambda XY (X^{d-2} - c^{d-2} Y^{d-2})$ $(\lambda \in \mathbb{C}^*)$.
Thus it is clear that $C$ is projectively equivalent to the curve defined by $Z^d + XY(X^{d-2}+Y^{d-2}) = 0$. 

\vspace{1em}

\noindent
Case\,(ii) From Lemma \ref{lem:descendant} we know that $C$ is projectively equivalent to $F_d$, $F'_d$ or $F''_d$ in this case.

\vspace{1em}

\noindent
Case\,(iii) In this case $G$ is a subgroup of $\textup{Aut}(K_d)$.
Since $\textup{Aut}(K_d)$ has an odd order $3(d^2-3d+3)$, we know that $G = \textup{Aut}(K_d)$ by our assumption that $|G| > d^2$.
It follows from Proposition \ref{prop:uniquness_of_K_d} that $C$ is projectively equivalent to Klein curve $K_d$.
\end{proof}


\vspace{1em}

\noindent   
\textbf{Acknowledgments.} \addcontentsline{toc}{part}{Acknowledgments} \label{Acknowledgments}
The author expresses his sincere gratitude to his professor Kazuhiro Konno and Professor Akira Ohbuchi
for their constructive comments and warm encouragement.
He also thank Professor Homma for his comment on preceding results.



\end{document}